\documentclass[12pt]{amsart}
\usepackage{amsmath}
\usepackage{amsfonts}
\usepackage{amsthm}
\usepackage{amssymb}
\usepackage{tikz-cd}
\usepackage{graphics}
\usepackage{array}
\usepackage{dsfont}
\usepackage{color}
\usepackage{wasysym}
\usepackage{hyperref}
\usepackage{pdfsync}
\usepackage{mathrsfs}

\usepackage[alphabetic]{amsrefs}
\usepackage{geometry}
\usepackage{bbm}
\usepackage[all,cmtip]{xy}
\usepackage{tikz}\usetikzlibrary{intersections}

\newcommand{\N}{{\mathbb N}}
\newcommand{\Z}{{\mathbb Z}}
\newcommand{\C}{{\mathbb C}}
\newcommand{\Q}{{\mathbb Q}}

\DeclareMathOperator{\Pic}{Pic}

\DeclareMathOperator{\codim}{codim}

\DeclareMathOperator{\GV}{GV}
\DeclareMathOperator{\IT}{IT_0}

\theoremstyle{plain}

\newtheorem{thm}{Theorem}[section]

\newtheorem{coro}[thm]{Corollary}

\newtheorem*{Generic vanishing conjecture}{Generic vanishing conjecture}

\theoremstyle{definition}

\newtheorem{defi}[thm]{Definition}

\newtheorem*{ac}{Acknowledgements}

\newcommand{\cO}{\mathcal{O}}

\newcommand{\cI}{\mathcal{I}}

\numberwithin{equation}{section}

\begin{document}

\title{Generic vanishing conjecture for divisors}
	
	\author{Fanjun Meng}
    \address{Department of Mathematics, University of California San Diego, 9500 Gilman Drive \# 0112, La Jolla, CA 92093-0112, USA}
    \email{f2meng@ucsd.edu}

	\thanks{2020 \emph{Mathematics Subject Classification}: 14F17, 14K05.\newline
		\indent \emph{Keywords}: abelian varieties, generic vanishing, positivity.}

\begin{abstract}
Generic vanishing conjecture proposed by Pareschi and Popa is about the relationship between the existence of subvarieties of codimension $>1$ of principally polarized abelian varieties representing minimal cohomology class and the positivity of their ideal sheaves twisted by the principal polarization. Inspired by their formulation, we prove a codimension $1$ analog of generic vanishing conjecture.
\end{abstract}

	\maketitle
	 \tableofcontents

\section{Introduction}

We work over the field of complex numbers $\C$.

In \cite{PP08}, Pareschi and Popa propose the following conjecture regarding the relationship between the existence of subvarieties of codimension $>1$ of principally polarized abelian varieties representing minimal cohomology class and the positivity of their ideal sheaves twisted by the principal polarization, in analogy with the well-known equivalence between a subvariety in projective space being of minimal degree and its ideal sheaf being Castelnuovo--Mumford $2$-regular.

\begin{Generic vanishing conjecture}[\cite{PP08}*{Conjecture A}]
    Let $(A, H)$ be an indecomposable principally polarized abelian variety (ppav) of dimension $g$ and $X$ a geometrically nondegenerate closed reduced subscheme of $A$ with ideal sheaf $\cI_X$ of pure dimension $1\le d\le g-2$. The following are equivalent:
    \begin{enumerate}
         \item $X$ has minimal cohomology class, i.e. $[X]=\frac{H^{g-d}}{(g-d)!}$.
        \item $\cI_X(H)$ is a $\GV$-sheaf.
        \item $\cI_X(2H)$ satisfies $\IT$.
        \item $\cO_X(H)$ is M-regular, and $\chi(X, \cO_X(H))=1$.
        \item Either $(A, H)$ is the polarized Jacobian of a smooth projective curve of genus $g$ and $X$ is $+$ or $-$ an Abel-Jacobi embedded copy of $W_d(C)$, or $g=5$, $d=2$, $(A, H)$ is the intermediate Jacobian of a smooth cubic threefold and $X$ is $+$ or $-$ a translate of the Fano surface of lines.
    \end{enumerate}
\end{Generic vanishing conjecture}

Note that for a ppav $(A, H)$, we use $H$ to denote the theta divisor of itself. $\GV$-sheaves are sheaves satisfying an analogous condition to the famous generic vanishing result for canonical bundles by Green and Lazarsfeld \cites{GL87, GL91}. M-regular sheaves are defined by Pareschi and Popa in \cite{PP03}, and they lead to many important applications to the theory of abelian varieties and beyond. Both notions are about positivity of sheaves on abelian varieties and are useful in this setting, see Definition \ref{GVM} for details.

The equivalence between $(1)$ and $(5)$ is conjectured in low dimension by Beauville \cite{Bea82} and Ran \cite{Ran81} and in general by Debarre \cite{Deb95}. $(2)\Leftrightarrow(4)\Rightarrow(3)$ are known and straightforward. However, as stated in \cite{PP08}, it is not known that $(3)$ implies $(2)$ or $(4)$. It is known that $(5)\Rightarrow(4)$ for $W_d$ in a Jacobian by \cite{PP03} and for Fano surface of lines of a smooth cubic threefold by \cite{Hor07}. In particular, $(5)$ is known to imply $(1)$, $(2)$, $(3)$ and $(4)$. It is known that $(2)\Rightarrow(1)$ by \cite{PP08}. The equivalence of $(1)$, $(2)$, $(4)$ and $(5)$ is known when $d=1$ by Matsusaka--Ran criterion and \cite{PP08} and when $g=4$ by \cites{Ran81, PP08}. Moreover, $(2)\Rightarrow(5)$ is additionally known when $d=g-2$ by \cite{PP08} and when $g=5$ by \cite{CMPS18}.

In this paper, we aim to prove the codimension $1$ analog of generic vanishing conjecture. That is, we consider a divisor $D>0$ (not necessarily reduced or irreducible), i.e. an effective divisor $D\neq0$, on the ppav $(A, H)$ of arbitrary dimension, and show that $(1)$, $(2)$, $(3)$ and $(4)$ are equivalent in this setting. Note that $(5)$ is trivially false in this setting since there are indecomposable principally polarized abelian varieties which are not Jacobians or intermediate Jacobians and we can choose $D$ to be $H$. Thus we do not consider $(5)$ in the following.

\begin{thm}\label{Generic vanishing conjecture in codim 1}
     Let $(A, H)$ be an indecomposable principally polarized abelian variety and $D>0$ a divisor on $A$. The following are equivalent:
    \begin{enumerate}
        \item $D$ is a translate of $H$.
        \item $\cI_D(H)$ is a $\GV$-sheaf.
        \item $\cI_D(2H)$ satisfies $\IT$.
        \item $\cO_D(H)$ is M-regular, and $\chi(D, \cO_D(H))=1$.
        \end{enumerate}
\end{thm}

As pointed in \cite{PP08}, whether $(3)$ implies $(2)$ or $(4)$ in the original generic vanishing conjecture is open. In Theorem \ref{Generic vanishing conjecture in codim 1}, the main task is to show that $(3)$ implies $(2)$, which follows from the following stronger theorem.

\begin{thm}\label{2-1>0}
    Let $(A, H)$ be a principally polarized abelian variety and $D$ a nef divisor on $A$. If $\cO_A(2H-D)$ satisfies $\IT$, then $\cO_A(H-D)$ is a $\GV$-sheaf.
\end{thm}

For coherent sheaves on abelian varieties, M-regularity (or $\IT$) can be viewed as an analog of ampleness and GV property can be viewed as an analog of nefness. Thus the intuitive idea behind Theorem \ref{2-1>0} is simple. In arithmetic, given an integer $n$, if $2\cdot1-n>0$, then $1-n\ge0$. We can view the principal polarization $H$ on $A$ as the identity $1$ in arithmetic. However, this is oversimplified. For example, the identity is unique in arithmetic while an abelian variety can have different principal polarizations. Actually, the proof of Theorem \ref{2-1>0} is a tricky application of several Mumford's deep and classic results regarding line bundles on abelian varieties. Just as the above arithmetic inequality fails if $n$ is a rational number, Theorem \ref{2-1>0} fails if $D$ is a $\Q$-divisor. It also fails if $D$ is arbitrary. Consider an abelian variety $A$ with two principal polarizations $H$ and $H'$ such that they are not numerically equivalent to each other. Let $D=2H-H'$. Then $\cO_A(2H-D)=\cO_A(H')$ satisfies $\IT$ but $\cO_A(H-D)=\cO_A(-H+H')$ is not a $\GV$-sheaf. If $\cO_A(-H+H')$ were a $\GV$-sheaf, then $H'-H$ is nef by \cite{PP11b}*{Theorem 4.1}. This nefness implies that $H$ and $H'$ are numerically equivalent by calculating the top self-intersection number of $H'$ and using that $H$ and $H'$ are both principal polarizations, which is a contradiction.

\begin{ac}
    {I would like to thank Nathan Chen and Mihnea Popa for helpful discussions and comments. I am partially supported by the Simons Collaboration Grant on Moduli of Varieties.}
\end{ac}

\section{Preliminaries}

We work over the field of complex numbers $\C$. We recall several definitions first.

\begin{defi}\label{csl}
Let $\mathcal{F}$ be a coherent sheaf on an abelian variety $A$. The \emph{cohomological support loci} $V_l^i(A, \mathcal{F})$ for $i\in\N$ and $l\in\N$ are defined by
$$V_l^i(A, \mathcal{F})=\{\alpha\in\Pic^0(A)\mid\dim H^i(A, \mathcal{F}\otimes\alpha)\geq l\}.$$
We use $V^i(A, \mathcal{F})$ to denote $V_1^i(A, \mathcal{F})$.
\end{defi}

\begin{defi}\label{GVM}
A coherent sheaf $\mathcal{F}$ on an abelian variety $A$
\begin{enumerate}
	\item[(1)] is a GV-\emph{sheaf} if $\codim_{\Pic^0 (A)} V^i (A, \mathcal{F}) \ge i$ for every $i>0$.
	\item[(2)] is \emph{M-regular} if $\codim_{\Pic^0 (A)} V^i (A, \mathcal{F}) > i$ for every $i>0$.
	\item[(3)] \emph{satisfies} $\IT$ if $V^i (A, \mathcal{F})=\emptyset$ for every $i>0$.
\end{enumerate}
\end{defi}

It is known that M-regular sheaves are ample by \cite{Deb06}*{Corollary 3.2}, and GV-sheaves are nef by \cite{PP11b}*{Theorem 4.1}. A line bundle on $A$ is a $\GV$-sheaf if and only if it is nef by \cite{Hac04}*{Theorem 1.2} and \cite{PP11b}*{Theorem 4.1}. A line bundle on $A$ satisfies $\IT$ if and only if it is M-regular if and only if it is ample by \cite{Deb06}*{Corollary 3.2}.

\section{Proofs of main results}

\begin{proof}[Proof of Theorem \ref{2-1>0}]
Let $g=\dim A$. If $g=0$, then the conclusion is trivial. Thus we assume that $g>0$ in the following. Define
\[
 P(t):=\frac{(tH-D)^g}{g!}\in\Q[t].
\]
Since $H$ is a principal polarization, $H^g=g!$ and thus $P(t)$ is monic. We first show that $P(t)\in\Z[t]$. For an arbitrary divisor $E$, we define a homomorphism
\begin{align*}
    \phi_E\colon A&\to\hat A \\
 a&\mapsto T_a^*\cO_A(E)\otimes_{\cO_A}\cO_A(-E),
\end{align*}
 where $\hat A$ is the dual abelian variety of $A$, and $T_a$ is the translation morphism given by $a$. The homomorphism $\phi_H$ is an isomorphism since $H$ is a principal polarization, and hence
$f:=\phi_H^{-1}\circ\phi_D$ is an endomorphism of $A$. By \cite{Mum08}*{Theorem 4, page 180}, there exists a monic polynomial $Q_f(t)\in\Z[t]$ of degree $2g$ such that $Q_f(n)=\deg(n_A-f)$ for every integer $n$, where $n_A$ is multiplication by $n$ on $A$. We have that
\begin{align*}
    \phi_H\circ(n_A-f)=\phi_H\circ n_A-\phi_H\circ f=\phi_{nH}-\phi_D=\phi_{nH-D}.
\end{align*}
By \cite{Mum08}*{The Riemann--Roch Theorem, Page 150}, we have that 
\begin{align*}
    Q_f(n)=\deg(n_A-f)=\deg\phi_{nH-D}
       =\chi(A, \cO_A(nH-D))^2
       =P(n)^2.
\end{align*}
Thus $P(t)^2=Q_f(t)\in\Z[t]$. By Gauss's lemma, $P(t)\in\Z[t]$.

By \cite{Mum70}*{Theorem 2, Page 98}, all the roots of $P(t)$ are real numbers. Since $\cO_A(2H-D)$ satisfies $\IT$, $2H-D$ is ample by \cite{Deb06}*{Corollary 3.2}. Moreover, $H$ is ample and $D$ is nef by the assumptions. Thus all the roots of $P(t)$ lie in $[0,2)$. Decompose $P(t)=t^a(t-1)^bR(t)$, where $R(t)\in\Z[t]$ is monic and $R(0)R(1)\ne0$. If $R(t)$ had positive degree, its roots $\rho_1,\ldots,\rho_s$ would lie in $(0,2)$ and not be $1$. Therefore
\[
 0<|R(1)|=\prod_{i=1}^s|1-\rho_i|<1,
\]
which is a contradiction to $R(1)\in\Z$. Hence, for some integer $r$ with $0\le r\le g$,
\[
 P(t)=t^r(t-1)^{g-r}.
\]

For each integer $n\ge1$, let $M_n:=(n+1)H-nD$. Define
\[
 P_{M_n}(t):=\frac{(tH+M_n)^g}{g!}=n^gP\left(\frac{t+n+1}{n}\right)=(t+n+1)^{r}(t+1)^{g-r}.
\]
All the roots of $P_{M_n}(t)$ are negative, thus $H^i(A, \cO_A(M_n))=0$ for $i>0$ by \cite{Mum70}*{Theorem 2, Page 98}. Thus by \cite{Mum08}*{The Riemann--Roch Theorem, Page 150}, we have that
\[
 h^0(A,\cO_A(M_n))=\chi(A,\cO_A(M_n))=P_{M_n}(0)=(n+1)^r>0.
\]
Therefore $M_n$ is linearly equivalent to an effective divisor and thus is nef. Since
\[
 \frac{M_n}{n}=H-D+\frac{H}{n},
\]
let $n$ go to positive infinity, and we deduce that $H-D$ is nef. By \cite{Hac04}*{Theorem 1.2}, $\cO_A(H-D)$ is a $\GV$-sheaf.
\end{proof}

Note that in the above proof, if we consider directly $H-D$ and
\[
 P_{H-D}(t):=\frac{(tH+H-D)^g}{g!}=P(t+1)=(t+1)^{r}t^{g-r},
\]
we cannot get the desired conclusion. This is because $P_{H-D}(t)$ can have nonnegative roots and thus we cannot get $H^i(A, \cO_A(H-D))=0$ for $i>0$ by \cite{Mum70}*{Theorem 2, Page 98}. Indeed, these higher cohomology groups need not vanish. Next, we prove Theorem \ref{Generic vanishing conjecture in codim 1}.

\begin{proof}[Proof of Theorem \ref{Generic vanishing conjecture in codim 1}]
Statement $(2)$ implies $(3)$ since GV-sheaves are nef by \cite{PP11b}*{Theorem 4.1}. Since $D$ is an effective divisor on an abelian variety, $D$ is nef. Thus statement $(3)$ implies $(2)$ by Theorem \ref{2-1>0}. Statement $(2)$ is equivalent to $(4)$ by \cite{PP08}*{Lemma 3.3}. It is straightforward that $(1)$ implies $(2)$. The only remaining part is that $(2)$ implies $(1)$. Since $(A, H)$ is an indecomposable principally polarized abelian variety, any translate of $H$ is reduced and irreducible. Since $\cO_A(H-D)$ is a nonzero $\GV$-sheaf, $V^0(A, \cO_A(H-D))$ is nonempty by \cite{HPS18}*{Lemma 7.4}. Thus some translate $T$ of $H$ satisfies $T\sim D+E$ where $E$ is an effective divisor. Since $D>0$, and $T$ is a principal polarization and is reduced and irreducible, it follows that $T=D+E$ and thus $E=0$ and $T=D$.
\end{proof}

Next, we give an application of Theorem \ref{2-1>0}.

\begin{coro}
Let $A$ be an abelian variety with principal polarizations $H$ and $H'$ such that $2H-H'$ is nef. Then $H'$ is a translate of $H$.
\end{coro}

If we assume that $2H-H'$ is ample in the above statement, then it is not so difficult to prove it. By an inequality of Hodge type and a criterion for nef divisors on abelian varieties, we can deduce that $H'$ is a translate of $H$.

\begin{proof}
    Since $2H-H'$ is nef and $\cO_A(2H-(2H-H'))$ satisfies $\IT$, $\cO_A(H-(2H-H'))$ is a $\GV$-sheaf by Theorem \ref{2-1>0}. Thus $V^0(A,\cO_A(H'-H))$ is nonempty by \cite{HPS18}*{Lemma 7.4}. Thus some translate $T$ of $H$ satisfies $H'\sim T+E$ where $E$ is an effective divisor. Since both $H'$ and $T$ are principal polarizations and $E$ is effective and nef, we conclude that $E=0$ and $H'=T$.
\end{proof}

\bibliographystyle{amsalpha}
\bibliography{biblio}
\end{document}